\documentclass[11pt]{article}
\usepackage{amsmath, amssymb, amsthm}
\usepackage{enumitem}                                 
\usepackage[margin=1.0in]{geometry}
\usepackage[T1]{fontenc}

\AtEndDocument{\bigskip{\footnotesize%
\textsc{Department of Mathematics, The University of Haifa at Oranim, Tivon 36006, Israel} \par 
  \textit{E-mail address} \texttt{ amir.algom@math.haifa.ac.il}
}
 
\bigskip{\footnotesize%
\textsc{Department of Mathematics, the Pennsylvania State University, University
Park, PA 16802, USA} \par  
\textit{E-mail address} \texttt{ zhirenw@psu.edu
}
}}

\allowdisplaybreaks[1]

\newtheorem{theorem}{Theorem}[section]
\newtheorem{Remark} [theorem]{Remark}

\newtheorem{Counter-example}[theorem]{Counter example}
\newtheorem{Claim}[theorem]{Claim}

\newtheorem{Lemma}[theorem]{Lemma}

\newtheorem*{theorem*}{Theorem}

\newcommand{\ignore}[1]{}

\usepackage{xcolor}
\usepackage[pagebackref]{hyperref}
\hypersetup{
   colorlinks,
    linkcolor={red!60!black},
    citecolor={blue!60!black},
    urlcolor={blue!90!black}
}

\allowdisplaybreaks[1]

\title{Arbitrarily slow decay in the logarithmically averaged Sarnak conjecture}

\author{Amir Algom and Zhiren Wang}
\date{}

\begin{document}
\maketitle
\begin{abstract}
In 2017 Tao proposed a variant Sarnak's M\"{o}bius disjointness conjecture with logarithmic averaging: For any zero entropy dynamical system $(X,T)$,  $\frac{1}{\log N} \sum_{n=1} ^N \frac{f(T^n x) \mu (n)}{n}= o(1)$ for every  $f\in \mathcal{C}(X)$  and  every $x\in X$. We construct examples showing that this $o(1)$ can go to zero arbitrarily slowly. Nonetheless, all of our examples  satisfy the conjecture.
\end{abstract}
\section{Introduction}
 A topological dynamical system is a pair $(X,T)$ where  $X$ is  compact metric space and $T\in \mathcal{C}(X,X)$. If the system $(X,T)$ has zero topological entropy, then Sarnak's M\"{o}bius disjointness conjecture \cite[Main Conjecture]{Sarnak2012conjecture} asserts that 
\begin{equation} \label{Eq Sarnak}
\frac{1}{N}\sum_{n=1} ^N \mu(n)f(T^n x) = o(1),\quad \text{ for every } f\in \mathcal{C}(X) \text{ and  every } x\in X.
\end{equation}
Here $\mu:\mathbb N\to\{-1,0,1\}$ denotes the  M\"{o}bius function. We refer to the recent comprehensive surveys \cite{Fere2018Lem, Lem2021Kol, drmotasome} for references and reports on progress on the conjecture and related topics.

One strong piece of evidence towards the validity of Conjecture \eqref{Eq Sarnak} is that it is implied by the Chowla conjecture \cite[Conjecture 1.1]{Tao2017conjecture}, see  \cite{Tao2012Blog}. In 2017, Tao \cite{Tao2017conjecture} introduced logarithmically averaged versions of both Sarnak's and Chowla's conjectures, that were shown to be logically equivalent. Pertinent to this paper,  if  $(X,T)$ has zero entropy, then the logarithmically averaged  M\"{o}bius disjointness conjecture \cite[Conjecture 1.5]{Tao2017conjecture} predicts that 
\begin{equation} \label{Eq log Sarnak}
\frac{1}{\log N}\sum_{n=1} ^N \frac{\mu(n)f(T^n x)}{n} = o(1),\quad \text{ for every } f\in \mathcal{C}(X) \text{ and  every } x\in X.
\end{equation}
The surveys \cite{Fere2018Lem, Lem2021Kol, drmotasome} contain information on progress on this conjecture as well.  We also mention  the  work of Frantzikinakis  and Host \cite{Host2018Fran} on it.

Recently,  we studied the rate of decay in Sarnak's conjecture, showing that there are systems satisfying Sarnak's conjecture for which the  $o(1)$ as in \eqref{Eq Sarnak}  decays to zero arbitrarily slowly \cite{algom2022arbitrarily}. The purpose of this paper is to study this problem for the logarithmically averaged  M\"{o}bius disjointness conjecture \eqref{Eq log Sarnak}. Here is our main result:
\begin{theorem} \label{Main Theorem}
For every decreasing and strictly positive sequence  $\tau(n)\rightarrow 0$ there is a zero  entropy  dynamical system $(X,T)$  that satisfies:
\begin{enumerate}
\item There exist $x\in X$ and $f\in \mathcal{C}(X)$ such that $|f(x)|\equiv 1$ and
$$\liminf_{N\rightarrow \infty} \frac{1}{  \log \left( N \right)  \cdot  \tau(N)} \sum_{n=1} ^{N} \frac{\mu(n) f(T^n x) }{n} >0.$$ 

\item The  system $(X,T)$ satisfies  conjecture \eqref{Eq log Sarnak}.
\end{enumerate}
\end{theorem}

A few remarks are in order:
\begin{enumerate}
\item Using the summation by parts identity,  
$$  \sum_{n=1} ^{N} \frac{\mu(n) f(T^n x) }{n} \leq \sum_{M = 1} ^{N-1} \frac{1}{M+1} \left( \frac{1}{M} \sum_{n=1} ^M \mu(n) f(T^n x) \right) + \frac{1}{N} \sum_{n=1} ^N \mu(n) f(T^n x),$$
for any $x\in X$ and $f\in \mathcal{C}(X, \mathbb{R}_+)$ in a given dynamical system. So,  a decay rate in conjecture \eqref{Eq Sarnak} would imply a decay rate in conjecture \eqref{Eq log Sarnak}. Thus, as a corollary of Theorem \ref{Main Theorem} we see that there can be no decay rate in conjecture \eqref{Eq Sarnak}. Hence, Theorem \ref{Main Theorem} generalizes our previous result \cite[Theorem 1.1]{algom2022arbitrarily} about the (lack of a) decay rate in conjecture \eqref{Eq Sarnak}.

\item Theorem \ref{Main Theorem} part (1) is also formally stronger than \cite[Theorem 1.1 part (1)]{algom2022arbitrarily} since it is an assertion about the corresponding $\liminf$ rather than $\limsup$.

\item By \cite[Corollary 10]{Abdal2018Lem}, if Conjecture \eqref{Eq Sarnak} holds true then for every zero entropy system $(X,T)$ and $f\in \mathcal{C}(X)$, \eqref{Eq Sarnak} holds uniformly in $x\in X$. We remark that we expect a similar result to hold also for the logarithmically averaged version \eqref{Eq log Sarnak}, following e.g. the arguments of Gomilko-Lema\'{n}czyk-de la Rue \cite{Gomilko2020Lema}. We do not know, however, if this has appeared in print.
\end{enumerate}
Our results and methods are also related  to \cite{Fere2018Lem, kanigowski2020prime, fkl2021, Lian2021Shi, dolgopyat2020flexibility}. We refer to  \cite[Section 1]{algom2022arbitrarily} for more discussion about this.

We end this introduction with an outline of our construction. Morally, we take advantage of the logarithmic averaging  to run a version of our  argument from \cite{algom2022arbitrarily} in short intervals, thus obtaining stronger results. More precisely, we consider subshifts of 
$$\left(\lbrace -1,0,1 \rbrace^\mathbb{N} \times \lbrace -1,1 \rbrace^\mathbb{Z},\, T\right), \text{ where } T(y,z)=(\sigma y,\, \sigma^{y_1} z) \text{ and } \sigma \text{ is the left shift.}$$
 Given a rate function $\tau$ we first construct two  slowly growing sequences $q_k ^{(i)} \rightarrow \infty$, $i=0,1$. We then construct two subshifts such that their base comes from concatenating words of length $(k+1)^3-k^3$, that have non-zero entries at distance at least $q_k ^{(i)}$ from each other. Our space $X$ is  a product of these two spaces and a "switch" system: A subshift of $\lbrace 0,1 \rbrace^\mathbb{N}$ generated by elements $x$ satisfying that $x(i)=x(i+1)$ for  $k^3 \leq i < (k+1)^3 -1$. The function $f$ is taken to be 
$$f\left( (y^{(0)}, z^{(0)}), \, (y^{(1)}, z^{(1)}),\, s\right) = z_0 ^{(s_1)}.$$

For Theorem \ref{Main Theorem} part (1), our construction of the point $x\in X$ relies on the following  observation: For every $k\gg 1$, $i=0,1$ and some polynomially growing sequences $M_k ^{(i)}$, one may show that for some $c\in [0,q_k ^{(0)}]$ or $d\in [0,q_k ^{(1)}]$, either
$$   \sum_{b=c} ^{q_k ^{(0)}-1+c} \sum_{n=M_k ^{(0)}} ^{M_{k+1} ^{(0)} } \lambda(q_k ^{(0)}(n-M_k ^{(0)})+c+k^3) \cdot \mu(q_k ^{(0)}(n-M_k ^{(0)})+b+k^3)$$
or
$$  \sum_{b=d} ^{q_k ^{(1)}-1+d} \sum_{n=M_k ^{(1)}} ^{M_{k+1} ^{(1)}} \lambda(q_k ^{(1)}(n-M_k ^{(1)})+c+k^3) \cdot \mu(q_k ^{(1)}(n-M_k ^{(1)})+b+k^3)$$
are larger than $\frac{1}{ 2 q_{k}^{(0)}}$ (or respectively $\frac{1}{ 2 q_{k}^{(1)}}$) times their statistically expected values, up to a controllable error. Here $\lambda:\mathbb N\to\{\pm1\}$ denotes the Liouville function.

 We then construct our point $x$ via working in one of the subshifts in our space: For every $k$ we pick the $i$ giving the inequality above, specifying some digits of the  base point and an arithmetic progression in $\mu$ to put in some digits of the fiber. The value of the $s$ in the switch coordinate between $k^3$ and $(k+1)^3$ is determined by the $i$ giving this inequality. With some  work, we show that $x$ satisfies Theorem \ref{Main Theorem} part (1).

Finally, to derive part (2) of Theorem \ref{Main Theorem},  we show that the systems we construct have bounded measure complexity. This relies on the fact that every ergodic measure $\nu$ in the systems we construct is supported on a $T$-fixed point. The theorem then follows by invoking a recent result of Huang, Wang, and Ye \cite{Wang2019Ye}. See Section \ref{Part (2)} for more details.

\section{Proof of Theorem \ref{Main Theorem} Part (1)}

\subsection{Some preliminaries} \label{Section pre}
Let  $\sigma  $ denote the left shift on $\lbrace-1, 0 ,1\rbrace^\mathbb{Z}$ as well as on any of the following subspaces: $\lbrace-1, 0 ,1\rbrace^\mathbb{N}$, $\lbrace 0 ,1\rbrace^\mathbb{N}$, and $\lbrace-1,1\rbrace^\mathbb{Z}$. On $\lbrace-1, 0 ,1\rbrace^\mathbb{Z}$  we define the metric
$$ d(x,y) = 3^{- \min \lbrace |n|:\, x_n \neq y_n \rbrace}.$$
Also, for every $x\in \lbrace-1, 0 ,1\rbrace^\mathbb{N}$ and $k>l\in \mathbb{N}$ let $x|_l ^k \in \lbrace -1,0,1\rbrace^{k-l}$ be the word
$$x|_{l} ^k := (x_l,x_{l+1},....,x_k),$$
and we use similar notation in the space $\lbrace-1, 0 ,1\rbrace^\mathbb{Z}$ as well. Next, for every element $x\in \lbrace -1,0,1\rbrace^\mathbb{N}$ and $p\in \mathbb{N}_0$ we define $\sigma^{-p} x\in \lbrace -1,0,1\rbrace^\mathbb{N}$ as $\sigma^{-p} x = x$ if $p=0$, and otherwise
$$\left( \sigma^{-p} x \right)|_1 ^p = (0,...,0), \text{ and for all } n>p, \, \, \sigma^{-p} x (n) = x(n-p).$$ 
 Now, let
$$Z:=\lbrace-1, 0 ,1\rbrace^\mathbb{N} \times \lbrace-1,1\rbrace^\mathbb{Z}.$$
This is a metric space using the $\sup$-metric on both  coordinates.  Also, we denote by $\Pi_i$, $i=1,2$, the coordinate projections on $Z$.   We define the skew-product $T:Z\rightarrow Z$ via
$$T(y,z) = \left( \sigma\left(y\right),\, \sigma^{y_1} \left(z\right) \right).$$
We say that $X\subseteq Z$ is a subshift if it is closed and $T$-invariant.  The following Lemma  follows directly from our construction:
\begin{Lemma} \label{Lemma metric properties}
The  system $(Z,T)$ satisfies that for every $n\in \mathbb{N}$  and $x=(y,z)\in Z$
$$T^n (y,\,z) = \left( \sigma^n y,\, \sigma^{\sum_{i=1} ^n y_i} z \right).$$
\end{Lemma}
\subsection{Construction of some zero entropy systems to be used in the proof} \label{Section construction}
Fix a sequence $\tau(n) \rightarrow 0$ as in Theorem \ref{Main Theorem}. Assuming (as we may) that $\tau$ tends to zero sufficiently slowly, we construct a   sequence $q_{k} ^{(0)} = q_{k} ^{(0)}\left( \tau \right)\rightarrow \infty$ that satisfies the following properties:
\begin{enumerate}
\item $\frac{1}{2q_{k} ^{(0)}} > \tau(k^3)$, and
\item $q_{k} ^{(0)} < k^{ \frac{1}{8} }$.
\end{enumerate}
We also define a sequence $q_k ^{(1)}$ via 
$$q_k ^{(1)} := q_k ^{(0)}-1.$$
Note that we also have $\lim_{k\rightarrow \infty} q_{k} ^{(1)} =\infty$.

Next,  for every $k$ and $i\in \lbrace 0,1\rbrace$  let
$$A_k ^{(i)}  := \lbrace k^3+ j\cdot q_k ^{(i)} : \, j\in \mathbb{Z}_+,\, k^3 \leq k^3+ j\cdot q_k ^{(i)}  < (k+1)^3 \rbrace \subseteq \mathbb{N}.$$
For  every $k\in \mathbb{N}$ and $i\in \lbrace 0,1\rbrace$ we construct elements $s ^{(k,i)} \in \lbrace 0,1\rbrace^\mathbb{N}$ such that:
\begin{enumerate}

\item For every $k^3+j\cdot q_k ^{(i)}  \in A_k ^{(i)}$,
$$ s ^{(k,i)}  (k^3 + j\cdot q_k ^{(i)} )=1 \text{ if } j \leq \left[\frac{(k+1)^3-k^3 }{q_k ^{(i)} }\right]-1.$$

\item $ s ^{(k,i)}  (n) = 0$ for every integer $n\notin A_k ^{(i)}$,  or if $n \in A_k ^{(i)}$ but $n= k^3+j\cdot q_k ^{(i)}$ with \linebreak $j > \left[\frac{(k+1)^3-k^3 }{q_k ^{(i)} }\right]-1$.
\end{enumerate}
 The following Lemma is an immediate consequence of our construction. Recall the definition of $\sigma^{-p} x$ from Section \ref{Section pre}.
\begin{Lemma} \label{Lemma Tent strucutre}
For every $k\in \mathbb{N}$ large enough, $i\in \lbrace 0,1\rbrace$   and $p = 0,...,q_k ^{(i)}$   we have
$$\sum_{j \in  [k^3, \, (k+1)^3) \cap \mathbb{Z} } \left( \sigma^{-p}  s ^{(k,i)} \right)  (j ) =\left[\frac{(k+1)^3-k^3 }{q_k ^{(i)} }\right]-1.$$
\end{Lemma}

Next, for every  $k \in \mathbb{N}$ and $i\in \lbrace 0,1\rbrace$  define the truncations
$$R_k ^{(i)} = \left \lbrace \left( \sigma^{-p} s ^{(k,i)} \right)|_{ k^3 } ^{ (k+1)^3-1}  :\, p = 0,...,q_k ^{(i)}  \right \rbrace \subseteq \lbrace -1,0,1\rbrace^{ (k+1)^3 - k^3}.$$
We now define, for every $i\in \lbrace 0,1\rbrace$, the space $P ^{(i)}$ of all infinite sequences such that
$$P^{(i)}  = \lbrace y\in \lbrace -1, 0, 1 \rbrace^\mathbb{N}:\,   y|_{ k^3 } ^{ (k+1)^3-1} \in R^{(i)} _{k} \text{ for all } k\in \mathbb{N} \rbrace. $$
The following Lemma is an immediate consequence of Lemma \ref{Lemma Tent strucutre}, summation by parts, and the fact that the {C}es\`aro mean of the sequences $\frac{1}{q_k ^{(i)}}$ tends to $0$ for both $i=0,1$:
\begin{Lemma} \label{Lemma long tent}
For every $i\in \lbrace 0,1\rbrace$ and  $y\in P^{(i)}$,
$$\sum_{j=1} ^{k^3 -1} y(j) = \sum_{j \leq k} \left( \left[\frac{j^3-(j-1)^3 }{q_{j-1} ^{(i)} }\right] -1 \right) = o(k^3)$$
\end{Lemma}

Finally, recalling the definition of the system $(Z,T)$ from Section \ref{Section pre}, for every $i\in \lbrace 0 ,1\rbrace$ we define the subshift of $(Z,T)$
$$X_i = \text{cl} \left( \bigcup_{n\in \mathbb{N}_0}  T^n \left( P^{(i)} \times \lbrace -1,0,1\rbrace^\mathbb{Z} \right)  \right) .$$
For $j\in \lbrace -1,0,1\rbrace$ we denote by $\bar{j} \in \lbrace -1,0,1\rbrace^\mathbb{N}$ the constant element $\bar{j}(k)=i$ for every $k$.

\begin{Claim} \label{Zero entropy for each factor}
For every $i\in \lbrace 0,1\rbrace$ we have $h(X_i,\,T)=0$. 
\end{Claim}
\begin{proof}
Fix $i$.  We aim to prove the following statement: 
\begin{equation} \label{Eq Key claim 1}
\text{For every sequence } n_k\rightarrow \infty \text{ and } y\in P^{(i)}, \text{ if } \sigma^{n_k} y  \rightarrow y' \text{ then } \exists p \text{ such that } \sigma^p y' = \bar{0}.
\end{equation}
Note that \eqref{Eq Key claim 1} implies the Claim: Indeed, let $\nu$ be a $T$ ergodic invariant measure. Let $\nu_1$ be its projection to the first coordinate. Then, by \eqref{Eq Key claim 1} and the ergodic Theorem\footnote{For example, one can apply \cite[Exercise 2.3.7] {einsiedler2011ergodic} with \eqref{Eq Key claim 1} to see that there exists a $\nu_1$ generic point $y$ that admits a $p$ with $\sigma^p y = \bar{0}$.}, $\nu_1$ is the Dirac measure on $\lbrace \overline{0} \rbrace$. It follows that for  $\nu$-a.e. $(y,z)$, $T(y,z) =(y,z)$. This shows that $\nu$ has zero metric entropy, and the Claim follows from the variational principle \cite[Chapter 8]{Walters1982ergodic}.

To prove \eqref{Eq Key claim 1},  let $y\in P^{(i)}$. Suppose $\sigma^{n_k} y \rightarrow y'$. Let $m\in \mathbb{N}$. Then there is some $k_0$ such that for all $k>k_0$ we have that
$$y|_{n_k} ^{n_k +m} = y'|_1 ^m.$$
Note that, assuming $k$ is large enough (depending on $m$), there can be at most $2$ non-zero digit in $y|_{n_k} ^{n_k +m}$: Indeed, such entries appear in places of the form $\ell^3 + jq_\ell ^{(i)} + p$ for some $\ell$ and $p\in \lbrace 0,...,q_\ell ^{(i)} \rbrace$. We make the following two observations:
\begin{enumerate}
\item If there is some $\ell$ and $j_1<j_2$ such that 
$$n_k \leq \ell^3 +j_i q_\ell ^{(i)} +p \leq n_k +m,\quad i=1,2,$$
then
$$(j_2-j_1) \cdot q_\ell ^{(i)} \leq m.$$
So, assuming $\ell=\ell(k)$ is large enough, we see that $j_2-j_1 <1$, a contradiction.

\item If there is some $\ell$ and $j_1,j_2,p_1,p_2$ such that 
$$n_k \leq \ell^3 +j_1 q_\ell ^{(i)}+p_1 \leq n_k +m, \text{ and } n_k \leq (\ell+2)^3 +j_2 q_{\ell+2} ^{(i)}+p_2 \leq n_k +m$$
then, since $\ell^3 +j_1 q_\ell ^{(i)}+p_1 \leq (\ell+1)^3$ and $(\ell+2)^3 +j_2 q_{\ell+2} ^{(i)}+p_2\geq (\ell+2)^3$, we have that 
$$(\ell+2)^3 - (\ell+1)^3\leq (\ell+2)^3 +j_2 q_{\ell+2} ^{(i)}+p_2 - \left( \ell^3 +j_1 q_\ell ^{(i)}+p_1 \right)  \leq m.$$
Assuming $\ell$ is large enough, this is impossible. Note that the same argument works for with $\ell+2$ swapped for $\ell+a$ for any $a\geq 2$.

\end{enumerate}
We conclude that for every $m$ the word $y'|_1 ^m$ consists of $0$'s, with the exception of at most two non-zero entries (note that these non-zero entries must be the same regardless of $m$).  So, there exists some $j\in \mathbb{N}, j=j(y')$, such that  $\sigma^j y' = \bar{0}$, proving \eqref{Eq Key claim 1}.
\end{proof}

Finally, let
 $$A:=\lbrace w\in \lbrace 0,1 \rbrace^\mathbb{N}: \, w(i)=w(i+1),\,  k^3 \leq i < (k+1)^3 -1 \rbrace,$$
and define
$$\Sigma := \text{cl} \left( \bigcup_{l \in \mathbb{N}_0} \sigma^l A \right).$$
We require the following Lemma:
\begin{Lemma} \label{Lemma zero entropy}
$h(\Sigma, \sigma)=0$.
\end{Lemma}
\begin{proof}
As in the proof of Claim \ref{Zero entropy for each factor}, suppose we show that
\begin{equation} \label{Eq Key claim 2}
\text{For every sequence } n_k\rightarrow \infty \text{ and } w\in A, \text{ if } \sigma^{n_k} w \rightarrow w' \text{ then } \exists p \text{ such that } \sigma^p w' \text{ is a fixed point}.
\end{equation}
Then the Lemma will follow from the variational principle, since \eqref{Eq Key claim 2} implies that every ergodic measure for $(\Sigma,\sigma)$ is a Dirac mass on a $\sigma$-fixed point of the form $\bar{i}$ for $i=0,1$.

To prove \eqref{Eq Key claim 2}, suppose $\sigma^{n_k} w \rightarrow w'$ for some $w\in \Sigma$. Let $d(\cdot, \cdot)$ denote the usual distance between a point and a set in $\mathbb{R}$.  Let $\mathcal{C}$ denote the set of cubic positive integers. Fix $m\in \mathbb{N}$. Then there are two options:
\begin{enumerate}
\item If $\lim_{k\rightarrow \infty} d(n_k,\, \mathcal{C}) = \infty$ then there is some $k_0$ large enough such that for all $k>k_0$, $[n_k,n_k+m]$ does not contain a cubic number. 

\item Otherwise,   there is some $k_0 = k_0(m)$ such that for every $k>k_0$ the interval $[n_k,n_k+m]$ may contain at most $1$ cubic number. 
\end{enumerate}
Now, if (1) happens then for every $k$ large enough all the digits of $w|_{n_k} ^{n_k +m} = w'|_1 ^m$ are the same. If  (2) happens then still this might occur. Otherwise, there is some $j=j(w')$ where all the digits $w'|_1 ^j$ are equal, and then all the digits $w'|_{j+1} ^m$ are equal, but perhaps the constant digit occurring after $j$ differs from that occurring before $j$. Note that $j$ must be unique, and does not depend on $m$.  So, either $w'=\bar{i}$ for $i=0,1$ or $\sigma^j w' = \bar{i}$, proving \eqref{Eq Key claim 2}.
\end{proof}

\subsection{Correlations along arithmetic progressions in the  M\"{o}bius function} \label{Section correlations}
Recall the definition of $Z$ from Section \ref{Section pre} and let $g:Z\rightarrow \lbrace -1, 1\rbrace$ be the function
$$g(y,z) = z_0.$$
For every $k$, $i=0,1$,  and $r,c$ such that $r,c\in [0,q_{k} ^{(i)}]$, writing 
$$M_{k} ^{(i)}:= \sum_{j \leq k-1} \left( \left[\frac{j^3-(j-1)^3 }{q_{j-1} ^{(i)} }\right] -1 \right) $$
let
$$S_{r,c} ^{k, i}:=  \sum_{b=r} ^{q_k ^{(i)} -1+r} \sum_{n= M_{k} ^{(i)} } ^{ M_{k+1} ^{(i)} } \lambda (q_k ^{(i)} (n-M_k ^{(i)} )+c+k^3) \cdot \mu(q_k ^{(i)} (n-M_k ^{(i)} )+b+k^3).$$


In the following Lemma we use the construction from Section \ref{Section construction}.
\begin{Lemma} \label{Key Lemma}
For every $k$ and $i\in \lbrace 0,1\rbrace$ and  for every two integers $c,r \in [0,  q_{k} ^{(i)}]$,  let \newline $x \in P^{(i)} \times \lbrace -1,0,1\rbrace^\mathbb{Z} \subseteq X_i$ be any element such that:
\begin{enumerate}
\item For $k^3 \leq n < (k+1)^3$,  $\Pi_1 x (n) = s_k ^{(i)} (n-r)$.

\item For $M_k ^{(i)} \leq n < M_{k+1} ^{(i)}$,  $\Pi_2 x (n) = \lambda \left( q_{k} ^{(i)} (n-M_k ^{(i)})+c+k^3 \right)$.
\end{enumerate}
Then 
$$ \sum_{n=k^3} ^{(k+1)^3} g(T^n x ) \mu(n) = S_{r,c} ^{k,i } + O(q_k^{(0)}).$$
\end{Lemma} Note that by the construction of $P^{(i)} \times \lbrace -1,0,1\rbrace^\mathbb{Z}$ in Section \ref{Section construction}, there exists  an element $x$ as in the statement of the Lemma in that space.
\begin{proof}
In this proof we suppress the $i$ in our notation and simply write $q,\, M_k.$ 
First, for every two integers $j\in [M_k ,\, M_{k+1} ]$ and $b\in [r,\, q+r-1]$,
\begin{eqnarray*}
\sum_{d=1} ^{q(j-M_k)+b+k^3} \left( \Pi_1  x   \right) (d) &=& \sum_{d=1} ^{k^3-1} \left( \Pi_1  x  \right) (d)+ \sum_{d=k^3} ^{q(j-M_k)+b+k^3-1} \left( \Pi_1  x  \right) (d)\\
&=& M_k+ \sum_{d=k^3} ^{q(j-M_k)+b+k^3-1} s_{k} ^{(i)} (d-r)\\
&=&M_k+  \sum_{d=k^3-r} ^{q(j-M_k)+b+k^3-r-1} s_{k} ^{(i)} (d) =M_k+ j-M_k=j.\\
\end{eqnarray*}
Note the use of Lemma \ref{Lemma long tent} in the second equality, and the use of the definition of $s_k ^{(i)}$ together with the fact that $M_{k+1}-M_k = \left[ \frac{(k+1)^3-k^3 }{ q} \right]-1 $ in the last one. Therefore,
\begin{eqnarray*}& & \sum_{n=k^3} ^{(k+1)^3} g(T^n x) \mu (n)\\
 &= & \sum_{j=M_k} ^{M_{k+1} } \sum_{b=r} ^{q+r -1} g(T^{q\cdot (j-M_k)+b+k^3} x) \mu(q\cdot (j-M_k)+b+k^3) +O(q)\\
& = &\sum_{j=M_k} ^{ M_{k+1} } \sum_{b=r} ^{q+r -1} g\left( \sigma^{q\cdot (j-M_k)+b+k^3} \Pi_1 x, \, \sigma^{ \sum_{d=1} ^{q\cdot (j-M_k)+b+k^3} \left( \Pi_1  x  \right) (d)}  \Pi_2 x \right)   \mu(q\cdot (j-M_k)+b+k^3) \\
&  &\ \ \ + O(q)\\
&=& \sum_{b=r} ^{q+r -1} g\left( \sigma^{q\cdot (j-M_k)+b-r+k^3} s_{k} ^{(i)}, \, \sigma^{j} \Pi_2 x \right)   \mu(q(j-M_k)+b+k^3)  +   O(q) \\
&=&   \sum_{j=M_k} ^{ M_{k+1} } \sum_{b=r} ^{q+r -1} \lambda(q \cdot (j-M_k)+c+k^3) \cdot  \mu(q\cdot (j-M_k)+b+k^3) +O(q)\end{eqnarray*}

Indeed: The first equality follows since $g(T^n x)$ and $\mu$ are both bounded sequences, in the second equality we use Lemma \ref{Lemma metric properties}, and in the third equality we are using the previous equation array and the definition of $x$. This definition along with the definition of $s_k ^{(i)}$ justify the last equality, where we simply get the definition of $S_{r,c} ^{k ,i}$.
\end{proof}

\begin{Remark}  \label{Remark other possibility}
In the setup of Lemma \ref{Key Lemma}, we may similarly find another $x \in P^{(i)} \times \lbrace -1,0,1\rbrace^\mathbb{Z}$  that satisfies the conclusion  of Lemma \ref{Key Lemma}, but for  $-S_{r,c} ^{k, i}$. Indeed, this follows from the very same proof by picking $x \in P^{(i)} \times \lbrace -1,0,1\rbrace^\mathbb{Z}$ to be any element such that for every  $ k^3 \leq n<(k+1)^3$ we have $\Pi_1 x (n) = s_k ^{(i)} (n-r)$, and for $M_k ^{(i)} \leq n < M_{k+1} ^{(i)}$ we put $\Pi_2 x (n) = -\lambda \left( q_{k} ^{(i)} (j-M_k ^{(i)})+c+k^3\right)$.
\end{Remark}

We will also require the following Lemma:

\begin{Lemma} \label{Lemma intermediate}
For every $k$ large enough there  is either some $c\in [0 , q_{k} ^{(0)} )$ such that
$$
 S_{c,c}  ^{k, 0 } \geq \frac{1}{ 2 q_{k} ^{(0)} } \sum_{m=k^3} ^{(k+1)^3} \mu(m) \mu(m)- O(  q_{k} ^{(0)}),
$$
or some  $d\in [0 , q_{k} ^{(1)} )$ with
$$ -S_{d+1,d} ^{ k, 1} \geq \frac{1}{ 2 q_{k} ^{(1)} } \sum_{m=k^3} ^{(k+1)^3} \mu^2(m) - O(  q_{k} ^{(0)}).$$
\end{Lemma}
\begin{proof}
In this proof we again simply write  $q,\, M$ for $q_{k} ^{(0)}, M_{k} ^{(0)}$ respectively. Now, for every $c,r \in [0,  q]$,
$$\sum_{c=0} ^{q-1} S_{c+r,c} ^{k,0} = \sum_{m=1} ^{(k+1)^3-k^3} \lambda(m+k^3)\cdot \left( \mu(m+r+k^3)+...+\mu(m+r+q-1+k^3) \right)$$
$$+  O\left( \frac{q^2}{(k+1)^3-k^3} \right)$$
So,
$$   \sum_{c=0} ^{q-1}  S_{c,c} ^{k,0} = \sum_{m=1} ^{(k+1)^3-k^3} \lambda(m+k^3)\cdot \left( \mu(m+k^3)+...+\mu(m+q-1+k^3) \right)+ O\left( q^2 \right).$$
Similarly,
$$  \sum_{c=1} ^{q-1} S_{c+1,c} ^{k, 1 } = \sum_{m=1} ^{(k+1)^3-k^3} \lambda(m+k^3)\cdot \left( \mu(m+1+k^3)+...+\mu(m+q-1+k^3) \right)+  O\left( q^2 \right).$$

Combining these equations, 
$$    \sum_{c=0} ^{q-1}  S_{c,c} ^{k,0}  -    \sum_{d=1} ^{q-1} S_{d+1,d} ^{k, 1 }  = \sum_{m=k^3} ^{(k+1)^3} \lambda(m) \mu(m) +O\left( q^2 \right)= \sum_{m=k^3} ^{(k+1)^3} \mu^2(m) +O\left( q^2 \right). $$
This implies the Lemma.
\end{proof}

\subsection{Construction of the point and system as in Theorem \ref{Main Theorem}} \label{Section main proof}
Recall that for every $k$ the  inequality from Lemma \ref{Lemma intermediate} is  given by $q_{k} ^{(i)}$ where $i$ is either $0$ or $1$. Recalling the spaces constructed in Section \ref{Section construction}, we define
\begin{equation}\label{EqExample} X:= X_0 \times X_1 \times \Sigma. \end{equation}
We now construct a point $x\in X$ as follows: For every  $k\in \mathbb{N}$ and $k^3 \leq n < (k+1)^3$, let $i=0,1$ correspond to the term  yielding the  inequality from Lemma \ref{Lemma intermediate}. We put $\ell:=i$ and  then define in the $\ell$-th coordinate $x^{(\ell)}(n):=x(n)$ where $x$ is as in Lemma \ref{Key Lemma} if $i= 0$ or Remark \ref{Remark other possibility} if $i =1$, corresponding to $k$, and either $r=c$ and $c$ (if $i=0$) or $r=d+1$ and $c=d$ (if $i=1$) yielding the inequality from Lemma \ref{Lemma intermediate}. Also, for the indices $k^3 \leq n < (k+1)^3$ we put $\ell$ in the $\Sigma$-coordinate of $x$.   For all $\ell=0,1$ and  digits not covered by the procedure above, we make some choice that ensures $x^{(\ell)} \in P^{(\ell)} \times \lbrace -1,0,1\rbrace^\mathbb{Z}$. Note that by Lemma  \ref{Key Lemma}  and the construction of $P^{(\ell)}$, such a choice is readily available.

Finally, we make $X$ a dynamical system via the self-map $\hat{T} \in \mathcal{C}(X)$ defined by
$$\hat{T}(p^{(0)},\, p^{(1)},\, s) =  (T p^{(0)},\, T p^{(1)},\, \sigma(s)).$$
The function $f\in \mathcal{C}(X)$ is taken to be 
$$f( (y^{(0)}, z^{(0)}), \, (y^{(1)}, z^{(1)}),\, \, s) = z_0 ^{(s_0)}.$$

We now prove part (1) of Theorem \ref{Main Theorem} via the following two claims:

\begin{Claim}
We have $h(X, \hat{T})=0$.
\end{Claim}
\begin{proof}
By Claim \ref{Zero entropy for each factor} and Lemma \ref{Lemma zero entropy} each factor in the product space $X$ has zero entropy,  which implies the assertion via standard arguments. 
\end{proof}

\begin{Claim}
For all $N$ large enough,
$$  \sum_{n=1} ^{N} \frac{f(\hat{T}^n x) \mu(n)}{n} \geq \tau(N) \sum_{n=1} ^{N}  \frac{\mu^2(n)}{n}-O(1),$$
where $O(1)$ does not depend on $N$.  In particular,
$$\liminf_{N\rightarrow \infty} \frac{1}{ \left( \log N \right) \cdot  \tau(N)} \sum_{n=1} ^{N} \frac{f(\hat{T}^n x) \mu(n)}{n} \geq \frac{6}{\pi^2}.$$ 
\end{Claim}
\begin{proof}
In this proof whenever we write $q_k$ we mean $q_k ^{(0)}$ (this is of little consequence since $q_k ^{(1)} = q_k ^{(0)} -1$). First, we claim that for every large enough $k\in \mathbb{N}$,
\begin{equation} \label{Eq key}
 \sum_{n=k^3} ^{(k+1)^3} f(\hat{T}^n x ) \mu(n) \geq  \frac{1}{2 q_{k}  }\sum_{n=k^3}^{(k+1)^3}\mu(n)^2 -O(q_k)  
.\end{equation}
Indeed, this follows since by our construction,
$$  \sum_{n=k^3} ^{(k+1)^3} f(\hat{T}^n x ) \mu(n)= \sum_{n=k^3} ^{(k+1)^3} g(\hat{T}^n x^{(\ell)} (n) ) \mu(n)$$
where $g$ and $x^{(\ell)} (n)$ are as in Lemma \ref{Key Lemma} (corresponding to the parameters as in the choice of $x$). Then \eqref{Eq key} follows directly from a combination of Lemma \ref{Key Lemma} and Lemma \ref{Lemma intermediate}, together with the construction of $x$ and of $q$.

Given $N$ let $N'$ be such that $(N')^3$ is the largest cube satisfying $(N')^3 \leq N$. Then
$$N^{\frac{1}{3}} -1 \leq N' \leq N^{\frac{1}{3}}.$$
And,
\begin{eqnarray*}
&&\sum_{k=1} ^{N} \frac{f(\hat{T}^n x) \mu(n)}{n} \\
&=& \sum_{k=1} ^{N'-1} \sum_{n=k^3} ^{(k+1)^3} \frac{f(\hat{T}^n x) \mu(n)}{n} + \sum_{n=(N')^3} ^{N} \frac{f(\hat{T}^n x) \mu(n)}{n}  \\
&=& \sum_{k=1} ^{N'-1} \sum_{n=k^3} ^{(k+1)^3} \frac{f(\hat{T}^n x) \mu(n)}{k^3} - O\left(\sum_{k=1} ^{N'-1} ((k+1)^3 - k^3) \big(\frac 1{k^3}-\frac 1{(k+1)^3}\big)\right)- O\left ( \sum_{n=(N')^3} ^{N}\frac 1n\right)
\end{eqnarray*}

We now make use of \eqref{Eq key} and get 
\begin{eqnarray*}
&&\sum_{k=1} ^{N} \frac{f(\hat{T}^n x) \mu(n)}{n} \\
&\geq & \sum_{k=1} ^{N'-1}\left( \sum_{n=k^3} ^{(k+1)^3}  \frac{\mu^2(n)}{2q_k k^3} - O\left(\frac {q_k^2}{k^3}\right) \right)+ O\left(\sum_{k=1} ^{N'-1} \frac 1{k^2}\right)-  O\left( \log N-\log (N')^3 \right)\\ 
&\geq & \frac 1{2q_{N'-1}} \sum_{k=1} ^{N'-1} \sum_{n=k^3} ^{(k+1)^3}  \frac{\mu^2(n)}{  k^3}- O\left(\sum_{k=1} ^{N'-1}\frac{q_k^2} {k^3}\right) -O\left(\sum_{k=1} ^{N'-1}\frac 1{k^2}\right)-O\left( \log \frac{N}{ (N')^3} \right)\\
&\geq & \frac 1{2q_{N'-1}} \sum_{k=1} ^{N'-1} \sum_{n=k^3} ^{(k+1)^3}  \frac{\mu^2(n)}{n}-O(1)\\
&\geq&  \tau(N) \sum_{k=1} ^{N'-1} \sum_{n=k^3} ^{(k+1)^3}  \frac{\mu^2(n)}{n}-O(1)\\
&\geq&  \tau(N)  \sum_{n=1} ^{N}  \frac{\mu^2(n)}{n}-O\left(\tau(N)\sum_{n=(N')^3+1}^N\frac 1n\right)-O(1)\\
&\geq & \tau(N)  \sum_{n=1} ^{N}  \frac{\mu^2(n)}{n}-O(1).
\end{eqnarray*}
Note that we made of the facts that $q_k\leq k^{\frac18}$ and $\frac{1}{2q_k} > \tau(k^3)$ in the computations. The proof of the Claim, and thus of Theorem \ref{Main Theorem}, follows immediately by the standard fact that  $\lim_{N\to\infty}\frac 1{\log N} \sum_{n=1} ^{N}  \frac{\mu^2(n)}{n}=\frac6{\pi^2}$.

\end{proof}

\section{Proof of Theorem \ref{Main Theorem} Part (2)} \label{Part (2)}

In this Section we prove Part (2) of Theorem \ref{Main Theorem}. That is, we show that the system $(X, \hat{T})$ given in \eqref{EqExample} satisfies the logarithmically averaged  M\"{o}bius disjointness conjecture  \eqref{Eq log Sarnak}. In fact, we will  prove a stronger claim, that $(X, \hat{T})$ satisfies the "usual" M\"{o}bius disjointness conjecture \eqref{Eq Sarnak}.

To this end, we will invoke the following (special case of a) Theorem of Huang, Wang ,and Ye \cite{Wang2019Ye}: Let $\rho$ be an invariant measure for $(X, \hat{T})$. Letting $d$ be the $\sup$ metric on $X= X_0 \times X_1 \times \Sigma$, for every $n$ define a metric on $X$ via
$$\bar{d} (x,y)= \frac{1}{n} \sum_{i=0} ^{n-1} d(\hat{T} ^i x, \hat{T} ^i y).$$
Let $\epsilon>0$ and let
$$S_n(d,\rho,\epsilon) = \left\lbrace \min m: \exists x_1,...,x_m \text{ s.t. } \rho\left( \bigcup_{i=1} ^m B_{\bar{d}} (x_i,\epsilon) \right)>1-\epsilon \right\rbrace.$$
We say that $\rho$ has bounded measure complexity if for every $\epsilon>0$ we have that $S_n(d,\rho,\epsilon)=O_{\epsilon,\rho} (1)$.
\begin{theorem} \cite[Theorem 1.1]{Wang2019Ye} \label{Theorem zhiren}
If  every invariant measure $\rho$ has bounded measure complexity then $(X, \hat{T})$ satisfies the M\"{o}bius disjointness conjecture \eqref{Eq Sarnak}.
\end{theorem}
Thus, via Theorem \ref{Theorem zhiren}, if we prove the following Claim then Theorem \ref{Main Theorem} part (2) will follow:
\begin{Claim} \label{Claim finite}
The  system $(X, \hat{T})$ has bounded measure complexity.
\end{Claim}
\begin{proof}
Recall that 
$$X = X_0 \times X_1 \times \Sigma$$
where all these spaces were constructed in Section \ref{Section construction}. We now show 
 that \eqref{Eq Key claim 1} and \eqref{Eq Key claim 2}, that were already proved in Section \ref{Section construction}, imply the Claim: Indeed, let $\nu$ be a $\hat{T}$ ergodic invariant measure. Let $\tilde{\nu}$ be its projection to $\Pi_1 X_1 \times \Pi_1 X_2 \times \Sigma$, where $\Pi_1$ is the projection to the first coordinate. Then, by \eqref{Eq Key claim 1} and \eqref{Eq Key claim 2}  and the ergodic Theorem, $\tilde{\nu}$ is the Dirac measure on $\lbrace \overline{0} \rbrace \times \lbrace \overline{0} \rbrace \times \lbrace \overline{i} \rbrace$ for some $i\in \lbrace 0,1\rbrace$. It follows that for  $\nu$-a.e. $x$ we have that $\hat{T}x =x$. Therefore, for any invariant measure $\rho$ we have that for  $\rho$-a.e. $x$, $\hat{T}x =x$.  This clearly implies that $\rho$ has bounded measure complexity, as claimed.

\end{proof}
Thus, Theorem \ref{Main Theorem} part (2) is proved.

Finally, we make two more remarks. First, in \cite{Haung2021Ye} it is shown that systems with bounded measure complexity have zero entropy. So, for an alternative proof of  Claims \ref{Zero entropy for each factor} and \ref{Lemma zero entropy} we could have argued (as we do above) that \eqref{Eq Key claim 1} and \eqref{Eq Key claim 2} imply bounded measure complexity, and then appeal directly to \cite{Haung2021Ye}. Also, via Claims \ref{Zero entropy for each factor} and \ref{Lemma zero entropy} and their proof, an alternative proof of Theorem \ref{Main Theorem} part (2) may be derived from \cite{MR16} using the arguments presented in \cite[Section 3.4.1]{Fere2018Lem}. 

\section{Acknowledgements}
This research was supported by Grant No. 2022034 from the United States - Israel Binational Science Foundation 
(BSF), Jerusalem, Israel. 
Z.W. was also supported by the NSF grant DMS-1753042 and a von Neumann Fellowship at the IAS.

\bibliography{bib}
\bibliographystyle{plain}

\end{document}